\documentclass{amsart}
\usepackage{amssymb}

\input xy
\xyoption{all}

\newcommand{\AR[1]}{\mathsf{AR}[#1]}
\newcommand{\OG}{\mathsf{OG}}

\newcommand{\AbE[1]}{\mathsf{AE}(#1)}
\newcommand{\Comp}{\mathsf{Comp}}
\newcommand{\Tych}{\mathsf{Tych}}
\newcommand{\Id}{\mathrm{Id}}
\newcommand{\IR}{\mathbb R}
\newcommand{\Lip}{\mathsf{Lip}}
\newcommand{\win}{\underline{\mathrm{min}}}
\newcommand{\wax}{\underline{\mathrm{max}}}
\newcommand{\wa}{\pm}
\newcommand{\ord}{<}
\newcommand{\supp}{\mathrm{supp}}
\newcommand{\e}{\varepsilon}
\newcommand{\w}{\omega}
\newcommand{\pr}{\mathrm{pr}}
\newcommand{\h}{{\underline{\centerdot}}}
\newcommand{\id}{\mathrm{id}}
\newcommand{\word}{\sqsubset}%\eqcirc}
\newcommand{\sht}{\mathrm{ht}}
\newcommand{\A}{\mathcal A}
\newcommand{\U}{\mathcal U}

\newcommand{\IN}{\mathbb N}

\newcommand{\II}{\mathbb I}
\newcommand{\JJ}{\mathbb J}
\newcommand{\Ra}{\Rightarrow}
\newcommand{\conv}{\mathrm{conv}}

\newtheorem{theorem}{Theorem}[section]
\newtheorem{corollary}[theorem]{Corollary}
\newtheorem{problem}[theorem]{Problem}
\newtheorem{proposition}[theorem]{Proposition}
\theoremstyle{definition}
\newtheorem{definition}[theorem]{Definition}
\newtheorem{remark}[theorem]{Remark}

\title[$F$-Dugundji spaces, $F$-Milutin spaces and absolute $F$-valued retracts]{$F$-Dugundji spaces, $F$-Milutin spaces\\ and absolute $F$-valued retracts}

\author[T.~Banakh]{Taras Banakh}
\address[T.~Banakh]{Ivan Franko National University of Lviv (Ukraine) and Jan Kochanowski University in Kielce (Poland)}
\email{t.o.banakh@gmail.com}

\author[T.~Radul]{Taras Radul}
\address[T.~Radul]{Kazimierz Wielki University, Bydgoszcz (Poland) and Ivan Franko National University of Lviv (Ukraine)}
\email{tarasradul@yahoo.co.uk}
\dedicatory{Dedicated to the memory of V.V.~Fedorchuk}

\begin{document}
\begin{abstract} For every functional functor $F:\Comp\to\Comp$ in the category $\Comp$ of compact Hausdorff spaces we define the notions of $F$-Dugundji and $F$-Milutin spaces, generalizing the classical notions of a Dugundji and Milutin spaces. We prove that the class of $F$-Dugundji spaces coincides with the class of absolute $F$-valued retracts. Next, we show that for a monomorphic continuous functor $F:\Comp\to\Comp$ admitting tensor products each Dugundji compact is an absolute $F$-valued retract if and only if the doubleton $\{0,1\}$ is an absolute $F$-valued retract if and only if some points $a\in F(\{0\})\subset F(\{0,1\})$ and $b\in F(\{1\})\subset F(\{0,1\})$ can be linked by a continuous path in $F(\{0,1\})$. We prove that for the functor $\Lip_k$ of $k$-Lipschitz functionals with $k<2$, each absolute $\Lip_k$-valued retract is openly generated. On the other hand the one-point compactification of any uncountable discrete space is not openly generated but is an absolute $\Lip_3$-valued retract. More generally, each hereditarily paracompact scattered compact space $X$ of finite scattered height $n=\sht(X)$ is an absolute $\Lip_k$-valued retract for $k=2^{n+2}-1$.
\end{abstract}

\maketitle

\section{Introduction}

A classical Tietze-Urysohn Theorem \cite[2.1.8]{Eng} says that each continuous function $f:X\to \mathbb R$ defined on a closed subset $X$ of a normal topological space $Y$ admits a continuous extension $\bar f:Y\to\mathbb R$. According to a classical theorem of Dugundji \cite{Dug}, for a closed subset $X$ of a metrizable topological space $Y$ and a locally convex linear topological space $Z$ there is a  linear operator $u:C(X,Z)\to C(Y,Z)$ extending each continuous function $f\in C(X,Z)$ to a function $\bar f\in C(Y,Z)$ with values in the closed convex hull $\overline{\mathrm{conv}}(f(X))$ of $f(X)$ in $Z$.
Operators with this property will be called {\em regular}.

Here by $C(X,Z)$ we denote the linear space of all continuous maps from $X$ to $Z$. The linear space $C(X,\mathbb R)$ of real-valued continuous functions on a topological space $X$ is usually is denoted by $C(X)$. If the space $X$ is compact, then the linear space $C(X)$ carries a structure of a Banach lattice with respect to the $\sup$-norm $\|f\|=\sup_{x\in X}|f(x)|$.

A natural temptation to unify Tietze-Urysohn and Dugundji Theorem fails as there are pairs $(X,A)$ of compact Hausdorff (and hence normal topological) spaces $A\subset X$ admitting no regular linear extension operator $u:C(A)\to C(X)$. This circumstance led A.~Pe\l czy\'nski \cite{Pel} to the idea of introducing the class of Dugundji compact spaces. Those are compact spaces $X$ admitting for each embedding $X\hookrightarrow Y$ into a compact Hausdorff space $Y$ a regular linear extension operator $u:C(X)\to C(Y)$.
\smallskip

The systematic study of the class of Dugundji compact spaces was started by A.~Pe\l czy\'nski in \cite{Pel}. Soon, it was realized that Dugundji compact spaces can be characterized as absolute $P$-valued retracts for the functor $P:\Comp\to\Comp$ of probability measures in the category $\Comp$ of compact Hausdorff spaces and their continuous maps. Let us recall that for a compact Hausdorff space $X$ its space of probability measures $PX$ is the subspace of the Tychonoff power $\IR^{C(X)}$ consisting of all regular linear functionals $\mu:C(X)\to\IR$ (the regularity of $\mu$ means that $\mu(f)\subset\overline{\conv}(f(X))$~). Each point $x\in X$ can be identified with the Dirac measure $\delta_x:C(X)\to\IR$, assigning to each function $f\in C(X)$ its value $f(x)$ at $x$. The assignment $x\mapsto\delta_x$ defines a canonical embedding $\delta:X\to PX$ of $X$ into its space of probability measures.
\smallskip

A compact Hausdorff space $X$ is called an {\em absolute $P$-valued retract} if  for each embedding $X\subset Y$ into a compact Hausdorff space $Y$ there is a continuous map $f:Y\to PX$ extending the canonical embedding $\delta:X\to PX$, see \cite{Fedmir} for more details.

A breakthrough in understanding the structure of Dugundji compacta was made by R.~Haydon \cite{Haydon} who proved that the class of Dugundji compacta coincides with the class $\AbE[0]$ of compact absolute extensors in dimension zero. We say that a topological space $X$ is an {\em absolute extensor in dimension} $n$ if each continuous map $f:B\to X$ defined on a closed subspace $B$ of a compact Hausdorff space $A$ of dimension $\dim(A)\le n$ admits a continuous extension $\bar f:A\to X$. By $\AbE[n]$ we shall denote the class of compact absolute extensors in dimension $n$.

\begin{theorem}[Haydon]\label{Dug} For a compact Hausdorff space $X$ the following conditions are equivalent:
\begin{enumerate}
\item $X$ is a Dugundji compact space;
\item $X$ is an absolute $P$-valued retract;
\item $X$ is an absolute extensor in dimension 0.
\end{enumerate}
\end{theorem}

The implication $(3)\Ra(1)$ of this theorem is usually proved with help of Milutin compact spaces, see \cite{FF}. Let us recall \cite{FF} that a compact Hausdorff space $X$ is {\em Milutin} if there is a continuous surjective map $f:K\to X$ from a Cantor cube $K=\{0,1\}^\kappa$, admitting a regular averaging operator $u:C(K)\to C(X)$, i.e., a regular linear operator such that $u(\varphi\circ f)=\varphi$ for any $\varphi\in C(X)$. In \cite{M} Milutin proved that the unit interval $\II=[0,1]$ is Milutin and derived from this fact that each Dugundji compact space is Milutin. The converse is not true as shown by the example of the hyperspace $\exp_2(\{0,1\}^{\aleph_2})$ which is Milutin but not Dugundji, see \cite[6.7]{FF}.

Theorem~\ref{Dug} shows that Dugundji compact spaces are tightly connected with the functor of probability measures $P$ (this was observed and widely exploited by \v S\v cepin in \cite{SCH}). The relations of the class of Dugundji spaces to some other functors was studied by Alkinson and Valov \cite{AV}, \cite{V}.

In this paper for any functional functor $F:\Comp\to\Comp$ we define the notions of $F$-Dugundji and $F$-Mulutin  compact spaces and will characterize these spaces in terms of extension and averaging operators between the spaces of continuous functions thus generalizing Theorem~\ref{Dug} to other functors. In particular, we shall prove that the class of $F$-Dugundji compact spaces coincides with the class of absolute $F$-valued retracts. In Sections~\ref{s3} and \ref{s4} for certain (concrete functional) functors we shall study the class of absolute $F$-valued retracts and its relation to the classes of Dugundji compact spaces and of openly generated compacta.

\section{$F$-Dugundji and $F$-Milutin compact spaces}

All topological spaces considered in this paper are assumed to be Hausdorff. Undefined notions from the theory of functors in the category $\Comp$ can be found in the monograph \cite{TZ}.

For a compact space $X$ by $C(X)$ we denote the Banach lattice of all continuous functions endowed with the norm $\|\varphi\|=\sup_{x\in X}|\varphi(x)|$. Any (not necessarily continuous) function $\mu:C(X)\to\IR$ will be called a {\em functional} on $C(X)$. The space $\IR^{C(X)}$ of all functionals will be endowed with the Tychonoff product topology. For every $x\in X$ the Dirac measure $\delta_X(x)\in\IR^{C(X)}$ is the functional assigning to each function $\varphi\in C(X)$ its value $\varphi(x)$ at $x$.

Any continuous map $f:X\to Y$ between compact spaces induces a linear operator $f^*:C(Y)\to C(X)$, $f^*:\varphi\mapsto \varphi\circ f$, between the corresponding function spaces, called the {\em dual of} $f$. The {\em second dual operator} of $f$ is the function $f^{**}:\IR^{C(X)}\to \IR^{C(Y)}$ assigning to each functional $\mu\in\IR^{C(X)}$ the functional $f^{**}(\mu)\in\IR^{C(Y)}$, $f^{**}(\mu):\varphi\mapsto\mu(\varphi\circ f)$. Letting $\IR^{C(f)}:=f^{**}$, we can consider the construction $\IR^{C(\cdot)}:\Comp\to\Tych$ as a covariant functor from the category $\Comp$ to the category $\Tych$ of Tychonoff spaces and their continuous maps.
The functor $\IR^{C(\cdot)}$ contains the Dirac functor as a subfunctor. The {\em Dirac functor} $\delta:\Comp\to\Comp$ assigns to each compact space $X$ the closed subspace $\delta(X)=\{\delta_X(x):x\in X\}\subset \IR^{C(X)}$ consisting of the Dirac measures on $X$. It is clear that the Dirac functor $\delta$ is isomorphic to the identity functor.

The Tietze-Urysohn Theorem implies that for any injective continuous map $f:X\to Y$ between compacta the dual map $f^*:C(Y)\to C(X)$ is surjective and then the second dual map $f^{**}:\IR^{C(X)}\to \IR^{C(Y)}$ is a topological embedding. Consequently, for each closed subset $A\subset X$ and the identity embedding $i_A:A\to X$ we can identify the functional space $\IR^{C(A)}$ with the subspace $i_A^{**}(\IR^{C(A)})\subset \IR^{C(X)}$. The same convention will concern also other functors $F$: writing $a\in FA\subset FX$ we shall have in mind that $a\in Fi_A(FA)\subset FX$.

By a {\em functional functor} we shall understand any subfunctor $F:\Comp\to\Comp$ of the functor $\IR^{C(\cdot)}:\Comp\to\Tych$, containing the Dirac functor $\delta\subset F$.  Well-known examples of functional functors are the functor of probability measures $P$ \cite{Fedmir} and the functor of idempotent measures $I$ \cite{Zar}. Many known functors in the category $\Comp$ are isomorphic to functional functors, see \cite{Ra1}, \cite{Ra2}.

The notion of a regular operator $u:C(X)\to C(Y)$ appearing in the definitions of Dugundji and Milutin spaces is a partial case of the notion of an $F$-regular operator for a functional functor $F$. By an {\em operator} between function spaces $C(X)$ and $C(Y)$ we shall understand any (not necessarily linear or continuous) function $u:C(X)\to C(Y)$.
Each operator $u:C(X)\to C(Y)$ induces the dual operator $u^*:\IR^{C(X)}\to\IR^{C(Y)}$ assigning to each functional $\mu:C(Y)\to\IR$ the functional $u^*(\mu)\in\IR^{C(X)}$, $u^*(\mu):\varphi\mapsto \mu\circ u(\varphi)$. It is easy to check that the dual operator $u^*$ is continuous with respect to the Tychonoff product topologies on the functional spaces $\IR^{C(X)}$ and $\IR^{C(Y)}$.

\begin{definition} For a functional functor $F:\Comp\to\Comp$ and compact spaces $X,Y$, an operator $u:C(X)\to C(Y)$ is called {\em $F$-regular} if for any $y\in Y$ the functional $u^*(\delta_Y(y))$ belongs to $FX\subset \IR^{C(X)}$.
\end{definition}

Observe that an  operator $u:C(X)\to C(Y)$ is regular and linear if and only if it is $P$-regular for the functor of probability measures $P$.

Next we generalize the notions of  Dugundji and Milutin compact spaces introducing a functorial parameter is their definitions.

\begin{definition}Let $F:\Comp\to\Comp$ be a functional functor. A compact space $X$ is defined to be
\begin{itemize}
\item {\em $F$-Milutin} if there exists a surjective map $f:K\to X$ from a Cantor cube $K=\{0,1\}^\kappa$ admitting an $F$-regular averaging operator $u:C(K)\to C(X)$;
\item {\em $F$-Dugundji} if there exists an injective map $f:X\to K$ to a Tychonoff cube $K=[0,1]^\kappa$ admitting an $F$-regular extension operator $u:C(X)\to C(K)$.
\end{itemize}
\end{definition}

The notions of an extension and averaging operators is unified by the notion of an exave  operator introduced by Pe\l czy\'nski \cite{Pel}.

An operator $u:C(X)\to C(Y)$ is called an {\em $f$-exave} for $f$ if $f^*\circ u\circ f^*=f^*$. If $f$ is injective (resp. surjective), then the equality $f^*\circ u\circ f^*=f^*$ is equivalent to $f^*\circ u=\id_{C(X)}$ (resp. $u\circ f^*=\id_{C(Y)}$), in which case $u$ is called an {\em extension} (resp. {\em averaging}) operator for $f$.

\begin{theorem}\label{Rad} For a functional functor $F:\Comp\to\Comp$ and a map $f:X\to Y$ between compact Hausdorff spaces $X,Y$ the following conditions are equivalent:
\begin{enumerate}
\item there exists an $F$-regular exave operator $u:C(X)\to C(Y)$ for the map $f$;
\item there exists a continuous map $s:Y\to FX$ such that $Ff\circ s(y)=\delta_Y(y)$ for every $y\in f(X)\subset Y$.
\end{enumerate}
\end{theorem}

\begin{proof} To prove that $(1)\Ra(2)$, assume that $u:C(X)\to C(Y)$ is an $F$-regular exave operator $u:C(X)\to C(Y)$ for the map $f$. Then its dual operator $u^*:\IR^{C(Y)}\to \IR^{C(X)}$ is a continuous map such that $u^*\circ\delta_Y(Y)\subset FX$. Consider the map $s=u^*\circ\delta_Y:Y\to FX$ and take any point $y\in f(X)$. Choose any point $x\in f^{-1}(y)$. Since $u$ is an exave, $f^*\circ u\circ f^*=f^*$, which implies $f^{**}\circ u^*\circ f^{**}=f^{**}$ where $f^{**}:\IR^{C(X)}\to \IR^{C(X)}$ is the dual operator to $f^*:C(Y)\to C(X)$. Taking into account that $F$ is a subfunctor of the functor $\IR^{C(\cdot)}$, we conclude that $Ff=f^{**}|FX$. It is easy to check that $f^{**}(\delta_X(x))=\delta_Y(f(x))=\delta_Y(y)$, which implies that $$\delta_Y(y)=f^{**}(\delta_X(x))=f^{**}\circ u^*\circ f^{**}(\delta_X(x))=f^{**}\circ u^*(\delta_Y(y))=Ff\circ u^*(\delta_Y(y))=Ff\circ s(y).$$

To prove that $(2)\Ra(1)$, assume that $s:Y\to FX$ is a continuous map such that $Ff\circ s(y)=\delta_Y(y)$ for every $y\in f(X)\subset Y$. Define an operator $u:C(X)\to C(Y)$ assigning to each continuous function $\varphi\in C(X)$ the continuous function $u(\varphi):Y\to \IR$, $u(\varphi):y\mapsto s(y)(\varphi)$. The continuity of $u(\varphi)$ follows from the continuity of the function $s:Y\to FX\subset\IR^{C(X)}$ and the continuity of the evaluation operator
$\delta_\varphi:\IR^{C(X)}\to \IR$, $\delta_\varphi:\mu\mapsto\mu(\varphi)$.

Let us check that the operator $u:C(X)\to C(Y)$ is  $F$-regular. Consider the dual operator $u^*:\IR^{C(Y)}\to\IR^{C(X)}$ and fix any point $y\in Y$. Observe that the functional $u^*(\delta_Y(y))\in\IR^{C(X)}$ assigns to each function $\varphi\in C(X)$ the real number $\delta_Y(y)(u(\varphi))=u(\varphi)(y)=s(y)(\varphi)$, which implies that $u^*\circ\delta_Y(y)=s(y)\in FX$. This means that the operator $u$ is $F$-regular.

Finally, we check that $u$ is an exave for the map $f$, i.e., $f^*\circ u\circ f^*=f^*$ where $f^*:C(Y)\to C(X)$ is the dual operator induced by $f$. Given any function $\varphi\in C(Y)$ and any $x\in X$, we need to check that $f^*\circ u\circ f^*(\varphi)(x)=f^*(\varphi)(x)$.
For this let $y=f(x)$ and observe that
$$f^*\circ u\circ f^*(\varphi)(x)=u(\varphi\circ f)(f(x))=s(f(x))(\varphi\circ f)=Ff(s(f(x)))(\varphi)=\delta_Y(f(x))(\varphi)=\varphi(f(x))=f^*(\varphi)(x).$$ Hence $u$ is an exave for $f$ and the theorem is proved.
\end{proof}

Theorem~\ref{Rad} implies the following characterizations of $F$-Milutin and $F$-Dugundji compact spaces.

\begin{theorem}\label{cor2a} For a functional functor $F:\Comp\to\Comp$ and a compact Hausdorff space $X$ the following conditions are equivalent:
\begin{enumerate}
\item $X$ is $F$-Milutin;
\item there exist a surjective map $f:K\to X$ from a Cantor cube $K=[0,1]^\kappa$  and a continuous
map $s:X\to FK$ such that $Ff\circ s(x)=\delta_X(x)$ for every $x\in X$.
\end{enumerate}
\end{theorem}

\begin{theorem}\label{cor1a} For a functional functor $F:\Comp\to\Comp$ and a compact space $X$ the following conditions are equivalent:
\begin{enumerate}
\item $X$ is $F$-Dugundji;
\item every embedding $X\subset Y$ into a compact Hausdorff space $Y$ admits an $F$-regular extension operator $u:C(X)\to C(Y)$;
\item for every embedding $X\subset Y$ into a compact Hausdorff space $Y$ there exists a continuous map $s:Y\to FX$ such that $s(x)=\delta_Y(x)$ for all $x\in X\subset Y$;
\item for some  embedding $X\subset K$ into a Tychonoff cube $K=[0,1]^\kappa$ there exists a continuous map $s:K\to FX$ such that $s(x)=\delta_K(x)$ for all $x\in X\subset K$.
\end{enumerate}
\end{theorem}

\begin{proof} The equivalences $(2)\Leftrightarrow (3)$ and $(1)\Leftrightarrow (4)$ follow from Theorem~\ref{Rad} while $(3)\Ra(4)$ is trivial. To prove that $(4)\Ra(3)$, assume that
  for some  embedding $X\subset K$ into a Tychonoff cube $K=[0,1]^\kappa$ there exists a continuous map $s:K\to FX$ such that $s(x)=\delta_K(x)$ for all $x\in X\subset K$. Let $X\subset Y$ be any embedding into a compact space $Y$. By Tietze-Urysohn Theorem, the identity map $X\to X\subset K$ admits a continuous extension $f:Y\to K$. Then the map $s\circ f:Y\to FX$ has the required property: $s\circ f(x)=s(x)=\delta_Y(x)$ for every $x\in X$.
\end{proof}

\section{Absolute $F$-valued retracts and absolute $F$-Milutin spaces}\label{s3}

Observe that the last conditions in Theorems~\ref{cor1a}, \ref{cor2a} have sense for any (not necessarily functional) functor $F$. So, we have chosen these conditions as a base for the definitions of absolute $F$-valued retracts and absolute $F$-Milutin spaces. In the following definition, given a functor $F:\Comp\to\Comp$ and a closed subset $A\subset X$ of a compact Hausdorff space $X$ we write $a\in FA\subset FX$ for some element $a\in FX$ if $a\in Fi_A^X(FA)\subset FX$ where $i_A^X:A\to X$ is the identity embedding.

\begin{definition} Given a functor $F:\Comp\to\Comp$ we define a compact Hausdorff space $X$ to be
\begin{itemize}
\item an {\em absolute $F$-valued retract\/} if for any embedding $X\subset Y$ into a compact space $Y$ there is a map $r:Y\to FX$ such that $r(x)\in F(\{x\})\subset FX$ for every point $x\in X$;
\item an {\em absolute $F$-Milutin space} if  there are a surjective map $f:K\to X$ from a Cantor cube $K=\{0,1\}^\kappa$ and a map $s:X\to FK$ such that $s(x)\in F(f^{-1}(x))\subset FK$ for all $x\in X$.
\end{itemize}
\end{definition}

We shall say that a functor $F:\Comp\to\Comp$ {\em preserves preimages} ({\em over points}) if for any surjective map $f:X\to Y$ between compact spaces and any closed (one-point) set $A\subset Y$ we get $(Ff)^{-1}(FA)=F(f^{-1}(A))$ (which actually means that  $(Ff)^{-1}(Fi_A^Y(FA))=Fi_{f^{-1}(A)}^XF(f^{-1}(A))\subset FX$).

Theorems \ref{cor1a}, \ref{cor2a} imply the following corollary.

\begin{corollary}\label{c3.2} Let $F:\Comp\to\Comp$ be a functional functor preserving singletons. A compact space $X$ is
\begin{enumerate}
\item $F$-Dugundji if and only if $X$ is an absolute $F$-valued retract;
\item $F$-Milutin if $X$ is absolute $F$-Milutin.
\end{enumerate}
Moreover, if the functional functor $F$ preserves preimages over points, then a compact space $X$ is $F$-Milutin if and only if $X$ is absolute $F$-Milutin.
\end{corollary}

\begin{remark}
In \cite{AV}, \cite{V} partial cases of Corollary~\ref{c3.2} were proved for some concrete functional functors $F$.
\end{remark}

For a functor $F:\Comp\to\Comp$ by $\AR[F]$ we shall denote the class of all compact absolute $F$-valued retracts.
In the remaining part of the paper we shall address the following problem motivated by Theorem~\ref{Dug} and Corollary~\ref{c3.2}.

\begin{problem} Given a functor $F:\Comp\to\Comp$, detect compact spaces that belong to the class $\AR[F]$.
\end{problem}

This problem is not new and has been considered in \cite{SCH}, \cite{Cat98}, \cite{BBY}, \cite{V}.
An information about the classes $\AR[F]$ can be helpful because of the following simple fact.

\begin{proposition}\label{natural} Let $F,F':\Comp\to\Comp$ be two functors. If $F$ admits a natural transformation into $F'$, then $\AR[F]\subset\AR[F']$.
\end{proposition}

For normal functors $F$ the upper bound on the classes $\AR[F]$ was found by \v S\v cepin \cite{SCH}.

\begin{theorem}[\v S\v cepin]\label{shchep} $\AR[F]\subset\AbE[0]=\AR[P]$ for any normal functor $F:\Comp\to\Comp$.
\end{theorem}

Next, we find a condition on a functor $F$ guaranteeing that $\AbE[0]\subset\AR[F]$.

\begin{proposition}\label{p3.7} Let $F$ be a functor. If each Tychonoff cube is a absolute $F$-Mulutin space, then each Dugundji compact space is an absolute $F$-Milutin absolute $F$-valued retract. Consequently, $\AbE[0]\subset\AR[F]$.
\end{proposition}

\begin{proof} Let $X$ be a Dugundji compact space. To show that $X$ is an absolute $F$-Milutin absolute $F$-valued retract, fix any embedding $X\subset Y$ into a compact space $Y$. Let $Y\subset K$ be an embedding of $Y$ into a Tychonoff cube $K$. Since the Tychonoff cube $K$ is absolute $F$-Milutin, there exist a continuous map $g:C\to K$ from a Cantor cube $C=\{0,1\}^\kappa$
and  a continuous map $s:K\to FC$ such that $s(x)\in F(g^{-1}(x))\subset FC$ for each $x\in K$.

Consider the closed subset $Z=g^{-1}(X)$ in the Cantor cube $C$. By Haydon's Theorem~\ref{Dug},
 the Dugundji compact space $X$ is an absolute extensor in dimension zero. Consequently, the map $g|Z:Z\to X$ admits a continuous extension $\bar g:C\to X$. Consider the map $s|X:X\to FC$ and observe that for every $x\in X$ we get
 $$s(x)\in F(g^{-1}(x))\subset F(\bar g^{-1}(x))\subset FC.$$
 Consequently, the maps $\bar g:C\to X$ and $s|X:X\to FC$ witness that $X$ is absolute $F$-Milutin.

On the other hand, the map $r=F\bar g\circ s|Y:Y\to FX$ witnesses that $X$ is an absolute $F$-valued retract because
 $$r(x)=F\bar g(s(x))\in F\bar g(F(\bar g^{-1}(x)))\subset F(\{x\})\subset FX$$ for every $x\in X$.
\end{proof}

We shall say that a functor $F:\Comp\to\Comp$ admits {\em a tensor product} if for each cardinal $\kappa$ and each family of compacta $(X_\alpha)_{\alpha\in\kappa}$ there exists a continuous map $\otimes_{(X_\alpha)_{\alpha\in\kappa}}:\prod_{\alpha\in\kappa} F X_\alpha\to F(\prod_{\alpha\in\kappa} X_\alpha)$ which is natural by each argument. The latter means that for any maps $f_\alpha:X_\alpha\to Y_\alpha$, $\alpha\in\kappa$, between compact spaces and any $\beta\in\kappa$  the following diagram commutes:
$$\xymatrix{
&FX_\beta\\
\prod\limits_{\alpha\in\kappa}FX_\alpha
\ar^{\pr_\beta}[ru]\ar^{\bigotimes_{(X_\alpha)_{\alpha\in\kappa}}}[rr]
\ar_{\prod\limits_{\alpha\in\kappa}Ff_\alpha}[d]
&&F\big(\prod\limits_{\alpha\in\kappa}X_\alpha\big)
\ar_{F\pr_\beta}[lu]\ar^{F\big(\prod\limits_{\alpha\in\kappa}f_\alpha\big)}[d]\\
\prod\limits_{\alpha\in\kappa}FY_\alpha
\ar_{\bigotimes_{(Y_\alpha)_{\alpha\in\kappa}}}[rr]
&&F\big(\prod_{\alpha\in\kappa}Y_\alpha\big)
}
$$

According to \cite{tensorRad}, for any functors $F_1,F_2:\Comp\to\Comp$ admitting tensor products the composition $F_1\circ F_2$ admits a tensor product too.
It is known that each monadic functor (i.e., a functor that can be completed to a monad) admits a tensor product (see \cite[\S3.4]{TZ}).  In particular, the functors of probability measures $P$ \cite{Fedmir} and idempotent measures \cite{Zar}  are monadic and hence admit a tensor product. On the other hand, the functor $\exp\circ\exp$ of double hyperspace admits a tensor product but fails to be monadic, see \cite{tensorRad}.

\begin{theorem}\label{t3.8} Let $F:\Comp\to\Comp$ be a functor admitting a tensor product.
If the unit interval $\II=[0,1]$ is an absolute $F$-Milutin space, then every Tychonoff cube is an absolute $F$-Milutin space and $\AbE[0]\subset\AR[F]$.
\end{theorem}

\begin{proof} By Proposition~\ref{p3.7}, it suffices to prove that each Tychonoff cube $\II^\kappa$ is an absolute $F$-Milutin space.
Since $\II$ is an absolute $F$-Milutin space, there exist a surjective map $g:C\to\II$ from a Cantor cube $C$ and a map $s:\II\to FC$ such that $s(x)\in F(g^{-1}(x))\subset FC$ for all $x\in\II$.
Consider the $\kappa$-th power $\bar g:C^\kappa\to\II^\kappa$, $\bar g:(x_\alpha)_{\alpha\in\kappa}\mapsto (g(x_\alpha))_{\alpha\in\kappa}$, of the map $g$ and the $\kappa$-th power $\bar s:\II^\kappa\to (FC)^\kappa$ of the map $s$.

Fix a tensor product $\otimes$ for the functor $F$ and let $\otimes_{(C)_{\alpha\in\kappa}}:(FC)^\kappa\to F(C^\kappa)$ be its component.
We claim that the map  $r=\otimes_{(C)_{\alpha\in\kappa}}\circ\bar s:\II^\kappa\to F(C^\kappa)$ witnesses that the Tychonoff cube $\II^\kappa$ is absolute $F$-Milutin. Indeed, given any point $x=(x_\alpha)_{\alpha\in\kappa}\in\II^\kappa$, consider the element $\bar s(x)\in \prod_{\alpha\in\kappa}F(g^{-1}(x_\alpha))\subset\prod_{\alpha\in\kappa}FC$. By the naturality of the tensor product, $$r(x)=\otimes_{(g^{-1}(x_\alpha))_{\alpha\in\kappa}}\circ \bar s(x)\in F\Big(\prod_{\alpha\in\kappa}g^{-1}(x_\alpha)\Big)=F(\bar g^{-1}(x))\subset F(C^\kappa).$$
\end{proof}

In light of Theorem~\ref{t3.8}, it is natural to pose a problem of recognizing functors $F$ for which the unit interval $\II=[0,1]$ is absolute $F$-Milutin. We shall give an answer to this problem for functionally continuous monomorphic functors which admit a tensor product.

A functor $F:\Comp\to\Comp$ is called {\em functionally continuous} if for each compact spaces $X,Y$ the map $F_{X,Y}:C(X,Y)\to C(FX,FY)$, $F_{X,Y}:f\mapsto Ff$, is continuous. Here by $C(X,Y)$ we denote the space of all continuous functions from $X$ to $Y$, endowed with the compact-open topology.
By \cite[2.2.3]{TZ}, a monomorphic [epimorphic] functor $F:\Comp\to\Comp$ is functionally continuous if [and only if] it is continuous (i.e., preserves limits of inverse spectra). We recall that a functor $F:\Comp\to\Comp$ is {\em monomorphic} (resp. {\em epimorphic}) if $F$ preserves injective (resp. surjective) maps.

\begin{theorem}\label{t3.9} For a monomorphic functionally continuous functor $F:\Comp\to\Comp$, admitting a tensor product, the following conditions are equivalent:
\begin{enumerate}
\item the closed interval $[0,1]$ is an absolute $F$-Milutin space;
\item every Tychonoff cube is an absolute $F$-Milutin space;
\item $\AbE[0]\subset \AR[F]$;
\item the doubleton $\{0,1\}$ is an absolute $F$-valued retract;
\item there are points $a\in F(\{0\})\subset F(\{0,1\})$ and $b\in F(\{1\})\subset F(\{0,1\})$ which can be linked by a continuous path in $F(\{0,1\})$.
\item there is a point $a\in F(\{0\})$ such that for the embeddings $i_0:\{0\}\to\{0\}\subset\{0,1\}$ and $i_1:\{0\}\to\{1\}\subset\{0,1\}$ the points $Fi_0(a)$ and $Fi_1(a)$ can be linked by a continuous path in $F(\{0,1\})$;
\end{enumerate}
\end{theorem}

\begin{proof} We shall prove the implications $(1)\Ra(2)\Ra(3)\Ra(4)\Ra(5)\Ra(6)\Ra(1)$.
\smallskip

The implications $(1)\Ra(2)\Ra(3)$ have been proved in Theorem~\ref{t3.8} and Proposition~\ref{p3.7} while $(3)\Ra(4)$ is trivial.
\smallskip

$(4)\Ra(5)$ Assume that the doubleton $\{0,1\}$ is an absolute $F$-valued retract. Then there is a continuous map $\gamma:[0,1]\to F(\{0,1\})$ such that $\gamma(x)\in F(\{x\})\subset F(\{0,1\})$ for every $x\in\{0,1\}$. The map $\gamma$ can be considered as a continuous path in $F(\{0,1\})$ linking the points $a=\gamma(0)\in F(\{0\})\subset F(\{0,1\})$ and $b=\gamma(1)\in F(\{1\})\subset F(\{0,1\})$.
\smallskip

$(5)\Ra(6)$ Assume that some points $a_0\in F(\{0\})\subset F(\{0,1\})$ and $a_1\in F(\{1\})\subset F(\{0,1\})$ can be linked by a continuous path $\gamma:[0,1]\to F(\{0,1\})$ such that $\gamma(i)=a_i$ for each $i\in\{0,1\}$. Since $a_0\in F(\{0\})\subset F(\{0,1\})$, there is a point $a\in F(\{0\})$ such that $a_0=Fi_0(a)$ where $i_0:\{0\}\to\{0,1\}$ is the identity embedding.
Let $i_1:\{0\}\to\{1\}\subset\{0,1\}$ be the other embedding of $\{0\}$ into $\{0,1\}$.

Consider the retraction $r:\{0,1\}\to\{1\}\subset\{0,1\}$ and observe that the map $Fr:F(\{0,1\})\to F(\{1\})\subset F(\{0,1\})$ is a retraction of the space $F(\{0,1\})$ onto its subspace $F(\{1\})=i^{\{0,1\}}_{\{1\}}\big(F(\{1\})\big)$. It follows that  $\tilde\gamma=Fr\circ\gamma:\II\to F(\{1\})\subset F(\{0,1\})$ is a continuous path linking the points $\tilde\gamma(0)=Fr(a_0)=FrFi_0(a)=F(r\circ i_0)(a)=Fi_1(a)$ and $\tilde\gamma(1)=Fr(a_1)=a_1$. Joining the pathes $\gamma$ and $\tilde\gamma$ together, we can construct a continuous path in $F(\{0,1\})$ linking the points $Fi_0(a)=a_0$ and $Fi_1(a)$.
\smallskip

$(6)\Ra(1)$ Assume that for some $a\in F(\{0\})$ and the embeddings $i_0:\{0\}\to\{0\}\subset\{0,1\}$ and $i_1:\{0\}\to\{1\}\subset\{0,1\}$ the points $a_0=Fi_0(a)$ and $a_1=Fi_1(a)$ can be linked by a continuous path $\gamma:[0,1]\to F(\{0,1\})$ such that $\gamma(i)=a_i$ for $i\in\{0,1\}$.
For every compact space $X$ the functional continuity of the functor $F$ implies the continuity of the map $F_{X}:C(\{0,1\},X)\to C(F\{0,1\},FX)$, $F_{X}:f\mapsto Ff$, which implies the continuity of the map
$$\Phi_X:F(\{0,1\})\times C(\{0,1\},X)\to FX,\;\;\Phi_X:(c,f)\mapsto Ff(c).$$

Identify the square $X\times X$ of $X$ with the function space $C(\{0,1\},X)$, assigning to each pair $(x_0,x_1)\in X\times X$ the function $f_{x_0,x_1}:\{0,1\}\to X$, $f_{x_0,x_1}:i\mapsto x_i$. The continuity of the function $\Phi_X$ implies the continuity of the function
$$\Psi_X:X\times X\times \II\to FX,\;\;\Psi_X:(x,y,t)\mapsto \Phi_X(\gamma(t),f_{x,y})=Ff_{x,y}(\gamma(t)).$$

For every $x\in X$ consider the embeddings $f_x:\{0\}\to\{x\}\subset X$ and $Ff_x:F(\{0\})\to FX$.
For every $x,y\in X$ the equalities $f_{x,y}\circ i_0=f_x$ and $f_{x,y}\circ i_1=f_y$ imply
that
$$\Psi_X(x,y,0)=Ff_{x,y}(\gamma(0))=Ff_{x,y}Fi_0(a)=F(f_{x,y}\circ i_0)(a)=Ff_x(a)$$and
$$\Psi_X(x,y,1)=Ff_{x,y}(\gamma(1))=Ff_{x,y}Fi_1(a)=F(f_{x,y}\circ i_1)(a)=Ff_y(a).$$

For every $n\in\w$ and $k\in\{1,\dots,2^n\}$ consider the interval $\II_{n,k}=\big[\frac{k-1}{2^n},\frac{k}{2^n}\big]\subset\II$ and its closed neighborhood $\JJ_{n,k}=[0,1]\cap\big[\frac{k-1}{2^n}-\frac1{2^{n+2}},\frac{k}{2^n}+\frac1{2^{n+2}}\big]$.
We shall consider the finite family $\{\JJ_{n,k}\}_{k=1}^{2^n}$ as a compact space endowed with the discrete topology.

Let $$\JJ_{n}=\bigcup_{k=1}^{2^n}\JJ_{n,k}\times\{\JJ_{n,k}\}\subset \II\times\{\JJ_{n,k}\}_{k=1}^{2^n}$$ be the topological sum of the family $\{\JJ_{n,k}\}_{k=1}^{2^k}$ and $\pr_n:\JJ_n\to\II$, $\pr_n:(t,i)\mapsto t$, be the coordinate projection. For every $t\in\II$ let $k_0\le k_1$ be unique numbers such that $\{k_0,k_1\}=\{k\in\{1,\dots,2^n\}:t\in\JJ_{n,k}\}$ and let $t_0=(t,\JJ_{n,k_0})$, $t_1=(t,\JJ_{n,k_1})$ be the points composing the preimage $\pr_n^{-1}(t)=\{t_0,t_1\}$. Observe that for every\newline $t\in\II\setminus\bigcup_{k=1}^{2^n-1}(\JJ_{n,k}\cap \JJ_{n,k+1})$ the points $t_0$ and $t_1$ coincide.

Define a continuous map $s_n:\II\to F(\JJ_n)$ letting
$$s_n(t)=\begin{cases}
\Psi_{\JJ_n}\big(t_0,t_1,\frac{t-\min\JJ_{n,k+1}}{\max\JJ_{n,k}-\min\JJ_{n,k+1}}\big),&\mbox{if $t\in\JJ_{n,k}\cap \JJ_{n,k+1}$ for some $0<k<2^n$},\\
\Psi_{\JJ_n}(t_0,t_1,0)=\Psi_{X_n}(t_0,t_1,1),&\mbox{otherwise},
\end{cases}
$$
and observe that $s_n(t)\in F(\{t_0,t_1\})\subset F\JJ_n$ for every $t\in\II$.

In the Tychonoff product $\JJ_\w=\prod_{n\in\w}\JJ_n$ consider the closed subset
$$K=\{(x_n)\in \JJ_\w:\forall n,m\in\w\;\;\pr_n(x_n)=\pr_m(x_m)\}$$and let $\pr:K\to\II$ be the projection on $\II$ defined by $\pr((x_n)_{n\in\w})=\pr_1(x_1)$ for $(x_n)_{n\in\w}\in K$.
It can be shown that $K$ is a compact zero-dimensional space without isolated points, so $K$ is homeomorphic to the Cantor cube $\{0,1\}^\w$ according to the Brouwer theorem \cite[7.4]{Ke} characterizing the Cantor set. We claim that the map $\pr:K\to\II$ witnesses that the interval $\II=[0,1]$ is absolute $F$-Milutin.

By our assumption, the functor $F$ admits a tensor product $\otimes$. Let  $$\otimes_{(\JJ_n)_{n\in\w}}:\prod_{n\in\w}F\JJ_{n}\to F\big(\prod_{n\in\w}\JJ_n)=F\JJ_\w$$ be its component for the family $(\JJ_n)_{n\in\w}$.

 Consider the continuous map $s:\II\to  F\JJ_\w$ defined by $s(t)=\otimes_{(\JJ_n)_{n\in\w}}\big((s_n(t))_{n\in\w}\big)$.
Since the functor $F$ is monomorphic, for the identity embedding $i_K:K\to \JJ_\w$ the map $Fi_K:FK\to F\JJ_\w$ is a topological embedding, so we can (and will) identify $FK$ with its image $Fi_K(FK)$ in $F\JJ_\w$. We claim that for every $t\in\II$ the point $s(t)$ is contained in the set $F(\pr^{-1}(t))=F(\prod_{n\in\w}\pr_n^{-1}(t))\subset FK\subset F\JJ_\w$.

By the definition of the map $s_n$, for every $n\in\w$ the point $s_n(t)$ is contained in the set $F(\pr_n^{-1}(t))\subset F(\JJ_n)$ and then $(s_n(t))_{n\in\w}\in \prod_{n\in\w}F(\pr_n^{-1}(t))\subset \prod_{n\in\w}F(\JJ_n)$. Let $$\otimes_{(\pr_n^{-1}(t))_{n\in\w}}:\prod_{n\in\w}F(\pr_n^{-1}(t))\to F\big(\prod_{n\in\w}\pr_n^{-1}(t)\big)=F(\pr^{-1}(t))\subset FK$$ be the component of the tensor product for the family $(\pr_n^{-1}(t))_{n\in\w}$. The naturality of the tensor product guarantees that
$$s(t)=\otimes_{(\JJ_n)_{n\in\w}}\big((s_n(t))_{n\in\w}\big)=
\otimes_{(\pr^{-1}(t))_{n\in\w}}\big(s_n(t))_{n\in\w}\big)\in F\big(\prod_{n\in\w}\pr_n^{-1}(t)\big)=F(\pr^{-1}(t))\subset FK.$$
So, $s:\II\to FK$ is a well-defined continuous map witnessing that the closed interval $\II=[0,1]$ is an absolute $F$-Milutin space.
\end{proof}

Combining Theorem~\ref{t3.9} with Shchepin's Theorem~\ref{shchep}, we get:

\begin{corollary}\label{cor1} If a normal functor $F:\Comp\to\Comp$ admits a tensor product, then the equality $\AR[F]=\AbE[0]$ holds if and only if $\{0,1\}\in\AR[F]$.
\end{corollary}

Since each monadic normal functor admit a tensor product (see \cite[\S3.4]{TZ}), this corollary implies:

\begin{corollary}\label{mon} For a monadic normal functor $F:\Comp\to\Comp$ the equality $\AR[F]=\AbE[0]$ holds if and only if $\{0,1\}\in \AR[F]$.
\end{corollary}

Applying Corollary~\ref{mon} to the functor of idempotent measures $I$ (which  monadic and normal \cite{Zar}), we get the following corollary answering a problem posed in \cite{AV}.

\begin{corollary} The functor of idempotent measures $I$ has $\AR[I]=\AbE[0]$.
\end{corollary}

For non-normal functors $F$ the class $\AR[F]$ need not be contained in the class $\AbE[0]$ of Dugundji compacta. A typical example is the functor of superextension $\lambda$, see \cite{Ivanov}.

\begin{theorem}[Ivanov]\label{ivanov} The class $\AR[\lambda]$ of absolute $\lambda$-valued retracts coincides with the class of openly generated connected compacta.
\end{theorem}

This fact was crucial in the proof of the coincidence of the class $\OG$ of openly generated compacta  with the class $\kappa\mathsf M$ of $\kappa$-metrizable compacta, see \cite{SCH}. According to \cite{Shi} and \cite{SCH}, $\kappa$-metrizable compacta can be characterized as follows:

\begin{theorem}[Shirokov, \v S\v cepin]\label{shi} A compact Hausdorff space $(X,\tau_X)$ is $\kappa$-metrizable if and only if it is openly-generated if and only if for any embedding $X\subset Y$ into a compact Hausdorff space $(Y,\tau_Y)$ there is an  operator $e:\tau_X\to\tau_Y$ such that
\begin{itemize}
\item $e(U)\cap X=U$ for all open sets $U\subset X$ and
\item $e(U)\cap e(V)=\emptyset$ for any disjoint open sets $U,V\subset X$.
\end{itemize}
\end{theorem}

\section{Recognizing absolute $F$-valued retracts for some functional functors}\label{s4}

In this section we shall study the class $\AR[F]$ for functional functors $F$.
First we recall some properties of functionals.

Let $X$ be a compact space. We shall say that a functional $\mu:C(X)\to\IR$
\begin{itemize}
\item {\em preserves constants} if $\mu(c_X)=c$ for any constant function $c_X:X\to \{c\}\subset\IR$;
\item {\em preserves order} if $\mu(f)\le\mu(g)$ for any functions $f\le g$ in $C(X)$;
\item {\em weakly preserves order} if $\mu(a)\le\mu(f)\le\mu(b)$ for any function $f\in C(X)$ and constant functions $a,b\in C(X)$ with $a\le f\le b$;
\item {\em preserves minima} if $\mu(\min\{f,g\})=\min\{\mu(f),\mu(g)\}$ for any functions $f,g\in C(X)$;
\item {\em preserves maxima}  if $\mu(\max\{f,g\})=\max\{\mu(f),\mu(g)\}$ for any functions $f,g\in C(X)$;
\item {\em weakly preserves minima} if $\mu(\min\{f,c\})=\min\{\mu(f),\mu(c)\}$ for any $f\in C(X)$ and any constant function $c\in C(X)$;
\item {\em weakly preserves maxima} if $\mu(\max\{f,c\})=\max\{\mu(f),\mu(c)\}$ for any $f\in C(X)$ and any constant function $c\in C(X)$;
\item is {\em additive} if $\mu(f+g)=\mu(f)+\mu(g)$ for any functions $f,g\in C(X)$;
\item is {\em weakly additive} if $\mu(f+c)=\mu(f)+\mu(c)$ for any $f\in C(X)$ and  any constant function $c\in C(X)$;
\item is {\em weakly multiplicative} if $\mu(c\cdot f)=\mu(c)\cdot\mu(f)$ for any $f\in C(X)$ and  any constant function $c\in C(X)$;
%\item is {\em positively homogeneous} if  $\mu(c\cdot f)=c\cdot\mu(f)$ for any $c\in [0,\infty)$ and $f\in C(X)$;
%\item is {\em non-expanding} if  $|\mu(f)-\mu(g)|\le\sup_{x\in X}|f(x)-g(x)|$ for any functions $f,g\in C(X)$;
\item is {\em $k$-Lipschitz} for $k\ge 1$ if  $|\mu(f)-\mu(g)|\le k\cdot \|f-g\|$ for any functions $f,g\in C(X)$.
\end{itemize}
Here for a function $f\in C(X)$ by $\|f\|=\sup_{x\in X}|f(x)|$ we denote its norm in the Banach space $C(X)$.

These properties of functionals allow us to define subfunctors of the functor $\IR^{C(\cdot)}$ assigning to each compact space $X$ the following closed subspaces of the functional space $\IR^{C(X)}$:
\begin{itemize}
\item $\overline{V}(X)=\{\mu\in \IR^{C(X)}:\mbox{$\mu$ preserves constants}\}$;
\item $V_{\ord}(X)=\{\mu\in \overline{V}(X):\mu\mbox{ preserves order}\}$;
\item $V_{\word}(X)=\{\mu\in\overline{V}(X):\mu\mbox{ weakly preserves order}\}$;
\item $V_{\min}(X)=\{\mu\in \overline{V}(X):\mu\mbox{ preserves minima}\}$;
\item $V_{\max}(X)=\{\mu\in \overline{V}(X):\mu\mbox{ preserves maxima}\}$;
\item $V_{\win}(X)=\{\mu\in \overline{V}(X):\mu\mbox{ weakly preserves minima}\}$;
\item $V_{\wax}(X)=\{\mu\in \overline{V}(X):\mu\mbox{ weakly preserves maxima}\}$;
\item $V_{+}(X)=\{\mu\in \overline{V}(X):\mu\mbox{ is additive}\}$;
\item $V_{\pm}(X)=\{\mu\in \overline{V}(X):\mu\mbox{ is weakly additive}\}$;
\item $V_{\h}(X)=\{\mu\in \overline{V}(X):\mu\mbox{ is weakly multiplicative}\}$;
%\item $V_{\mathrm{ph}}(X)=\{\mu\in V(X):\mu\mbox{ is positively homogeneous}\}$;
%\item $\Lip_1(X)=\{\mu\in V(X):\mu\mbox{ is non-expanding}\}$;
\item $\Lip_k(X)=\{\mu\in \overline{V}(X):\mu\mbox{ is $k$-Lipschitz}\}$ for $k\ge 1$.
\end{itemize}

Many known functors can be written as intersections of the above functors. For example,
\begin{itemize}
\item $V=V_{\word}$ is the universal functor considered by T.Radul \cite{Rad1};
\item $P=V_{+}\cap V_{\h}\cap V_{\ord}$ is the functor of probability measures (see \cite{Fedmir});
\item $I=V_{\wa}\cap V_{\max}$ is the functor of idempotent measures (see \cite{Zar});
\item $O=V_{\wa}\cap V_{\ord}$ is the functor of weakly additive order-preserving functionals introduced by T.Radul \cite{Rad3};
\item $S=V_{\h}\cap V_{\wa}\cap V_{\ord}$ is the functor introduced and studied by V.Valov \cite{V};
\item $V_{\wax}\cap V_{\win}\cap V_{\ord}$ is isomorphic to the functor $G$ of growth hyperspaces (see \cite{Rad2});
\item $V_{\win}\cap V_{\max}\cap V_{\wa}$ and $V_{\min}\cap V_{\wax}\cap V_{\wa}$ are isomorphic to the hyperspace functor $\exp$ (see \cite{Rad4}).
\end{itemize}

Since the functors $V_{\win}\cap V_{\max}\cap V_{\wa}$ and $V_{\min}\cap V_{\wax}\cap V_{\wa}$ are isomorphic to the hyperspace functor $\exp$, Fedorchuk's description \cite{Fed} of the class $\AR[\exp]$ implies:

\begin{theorem} $\AR[V_{\win}\cap V_{\max}\cap V_{\wa}]=\AR[V_{\min}\cap V_{\wax}\cap V_{\wa}]=\AR[\exp]=\AbE[1]$.
\end{theorem}

The following description of the class $\AR[S]$ for the functor $S=\overline{V}\cap V_{\h}\cap V_{\wa}\cap V_{\ord}$ was obtained by V.Valov in \cite{V}.

\begin{theorem}[Valov]\label{valovS} For the functor $S=V_{\h}\cap V_{\wa}\cap V_{\ord}$ the class $\AR[S]$ coincides with the class $\OG$ of openly generated compacta.
\end{theorem}

In fact, the inclusion $\AR[S]\subset\OG$ can be derived from the following general theorem.

\begin{theorem}\label{main1} A compact space $X$ is openly generated if $X\in\AR[F]$ for some functional functor $F$ such that $$FX\subset \big\{\mu\in \overline{V}(X): \sup\{|\mu(f)-\mu(g)|:f,g\in C(X),\;f\cdot g=0,\; \|f-g\|\le 1\}<2\big\}.$$
\end{theorem}

\begin{proof} Fix any compact space $X\in\AR[F]$. If $X$ is empty, then there is nothing to prove. So, we assume that $X$ is not empty. To show that $X$ is openly generated, we shall apply the characterization Theorem \ref{shi}. Embed $X$ into any compact Hausdorff space $Y$. Since $X$ is an absolute $F$-valued retract, there is a map $r:Y\to FX$ such that $r(x)\in F(\{x\})\subset FX$ for every $x\in X$. This implies that $F(\{x\})\subset\overline{V}(x)$ is not empty and hence coincides with the singleton $\overline{V}(\{x\})=\{\delta_X(x)\}$. This allows us to identify $X$ with the closed subset $\{\delta_X(x):x\in X\}\subset FX\subset\overline{V}(X)$ consisting of Dirac measures.

First we construct an extension operator $e:\tau_X\to \tau_{FX}$. Choose a positive $\e>0$ such that $$2-2\e>\sup\{|\mu(f)-\mu(g)|:f,g\in C(X),\;fg=0,\; \|f-g\|\le 1\}.$$

Given an open set $U\subset X$ consider the open subset $e(U)$ of $FX$ that consists of all functionals $\mu\in FX\subset \IR^{C(X)}$ for which there is a continuous function $f:X\to[0,1]$ such that $\mu(f)>1-\e$, $\mu(-f)<-1+\e$ and $\supp(f)=f^{-1}(0,1]\subset U$. The complete regularity of $X$ implies that $e(U)\cap X=U$.

We claim that $e(U)\cap e(V)=\emptyset$ for any disjoint open sets $U,V\subset X$. Assuming the converse, we can find a functional $\mu\in e(U)\cap e(V)$ and two continuous functions $f_U,f_V:X\to[0,1]$ such that
\begin{enumerate}
\item $\supp(f_U)\subset U$, $\supp(f_V)\subset V$;
\item $\mu(f_U)>1-\e$ and $\mu(-f_V)<-1+\e$.
\end{enumerate}

The condition (1) implies that $f_U\cdot (-f_V)=0$, $\|f_U+f_V\|\le 1$ and hence $$|\mu(f_U)-\mu(-f_V)|\le 2-2\e$$by the choice of $\e$.
On the other hand, the condition (2) yields that $|\mu(f_U)-\mu(-f_V)|>2-2\e$ and this is a desired contradiction.

The operator $e:\tau_X\to\tau_{FX}$ and the map $r:Y\to FX$ generate the operator $$r^{-1}\circ e:\tau_X\to\tau_Y,\;\; r^{-1}\circ e:U\mapsto r^{-1}(e(U)),$$ which has two properties from Theorem~\ref{shi}. This theorem implies that the space $X$ is openly generated.
\end{proof}

Combining Theorem~\ref{main1} with Valov's Theorem~\ref{valovS}, we get:

\begin{corollary}\label{c4} For any $k\in[1,2)$ and any subfunctor $F\subset \Lip_k$ we get $\AR[F]\subset\OG$. If, moreover, $V_{\h}\cap V_{\wa}\cap V_{\ord}\subset F$, then $\AR[F]=\OG$.
\end{corollary}

This corollary implies the following result first established by Valov \cite{V}.

\begin{corollary} The functor $O=V_{\wa}\cap V_{\ord}$ of weakly additive order-preserving functionals has $\AR[O]=\OG$.
\end{corollary}

It is interesting to note that for $k\ge 3$ Corollary~\ref{c4} does not hold because the class $\AR[\Lip_3]$ contains some scattered compact spaces, which are not openly generated.

Let us recall that a topological space $X$ is called {\em scattered} if each non-empty subspace $A$ of $X$ contains an isolated point. For a subspace $A$ of $X$ by $A^{(1)}$ we denote the set of non-isolated points of $A$ and for every ordinal $\alpha$ define the $\alpha$-th derived set $X^{(\alpha)}$ of $X$ letting $X^{(0)}=X$, $X^{(1)}$ be the set of non-isolated points of $X$ and $$X^{(\alpha)}=\bigcap_{\beta<\alpha}(X^{(\beta)})^{(1)}$$for $\alpha>1$.
For a scattered space $X$ the transfinite sequence $(X_\alpha)_{\alpha}$ is strictly decreasing, so $X^{(\alpha)}=\emptyset$ for some $\alpha$. The {\em scattered height} $\sht(X)$ of $X$ is the  smallest ordinal $\alpha$ such that the $\alpha$th derived set $X^{(\alpha)}$ is finite. For example, an infinite compact space $X$ with a unique non-isolated point has scattered height $\sht(X)=1$.

A topological space $X$ is called {\em hereditarily paracompact} if each subspace of $X$ is paracompact. By \cite{BL}, a topological space is compact scattered and hereditarily paracompact if and only if it belongs to the smallest class of spaces containing the singletons and closed under taking one-point compactifications of topological sums of spaces from the class.

%\begin{theorem} For any scattered compact space $X$ of finite scattered height $s=\sht(X)$ and the functor $F=V_{\h}\cap V_{\wa}\cap V_{\word}\cap \Lip_{2^s-1}$ there is a continuous map $r:V_{\word}X\to FX$ such that $r\circ\delta_X=\delta_X$.

\begin{theorem}\label{main2} Each hereditarily paracompact scattered compact space $X$ of finite scattered height $s=\sht(X)$ with $|X^{(s)}|=1$ is an absolute $F_s$-valued retract for the functor $F_s=V_{\h}\cap V_{\wa}\cap V_{\word}\cap\Lip_{2^{s+1}-1}$.
\end{theorem}

\begin{proof} Since $V_{\word}X=VX=\prod_{f\in C(X)}[\min f,\max f]$ is an absolute retract, it suffices to construct a continuous map $p_X:VX\to F_{s}(X)$ such that $p_X\circ\delta_X=\delta_X$. The existence of such map $p_X$ will be proved by induction on the scattered height of $s=\sht(X)$. If $s=0$, then $X^{(0)}=X$ is a singleton and so is the space $F_{0}X$. Then the constant map $r_X:VX\to F_0X$ has the desired property: $p_X\circ\delta_X=\delta_X$.

Now we assume that for some $s\ge 1$ the existence of a map $p_K:VK\to F_{\sht(K)}K$ with $p_K\circ\delta_K=\delta_K$ has been proved for all hereditarily paracompact scattered compacta $K$ such that $\sht(K)<s$ and $|K^{(\sht(K))}|=1$. Assume that $X$ is a hereditarily paracompact scattered compact space of scattered height $\sht(X)=s$ and $X^{(s)}$ is a singleton $\{*\}$. By a result of Telgarsky \cite{Telga}, the scattered space $X\setminus\{*\}$, being paracompact, is strongly zero-dimensional, which allows us to find a disjoint cover $\A$ of $X\setminus\{*\}$ by non-empty compact open subsets of $X\setminus\{*\}$.
Each space $A\in\A$ has scattered height $\sht(A)<\sht(X)=s$. Decomposing $A$ into a disjoint finite union of open subsets, we can additionally assume that $A^{(\sht(A))}$ is a singleton for every $A\in\A$. The inductive assumption yields a continuous map  $p_A:VA\to F_{\sht(A)}(A)\subset F_{s-1}(A)$ such that $p_A\circ\delta_A=\delta_A$. Let $\chi_A:X\to\{0,1\}$ denote the characteristic function of the set $A$ (which means that $A=\chi^{-1}_A(1)$). Fix any retraction $r_A:X\to A$ and consider the retraction $Vr_A:VX\to VA$. For a functional $\mu\in VX$ put $\mu_A=p_A\circ Vr_A(\mu))\in F_{s-1}(A)$. Let $\A^\circ=\A\cup\{\emptyset\}$.

Define the map $p_X:VX\to \overline{V}X$ assigning to each functional $\mu\in VX$ the functional
$p_X(\mu)\in \overline{V}X$, which assigns to each function $\varphi\in C(X)$ the real number
$$p_X(\mu)(\varphi)=\varphi(*)+\sup_{A\in\A^\circ}\mu(\chi_A)\cdot \mu_A\big(\varphi|A-\varphi(*)\big)+
\inf_{A\in\A^\circ}\mu(\chi_A)\cdot \mu_A\big(\varphi|A-\varphi(*)\big).$$
In this formula for $A=\emptyset$ we assume that $\mu(\chi_A)\cdot \mu_A\big(\varphi|A-\varphi(*)\big)=0$, which implies that
$$\sup_{A\in\A^\circ}\mu(\chi_A)\cdot \mu_A\big(\varphi|A-\varphi(*)\big)=\max_{A\in\A^\circ}\mu(\chi_A) \cdot \mu_A\big(\varphi|A-\varphi(*)\big)\ge0$$ and
$$\inf_{A\in\A^\circ}\mu(\chi_A)\cdot \mu_A\big(\varphi|A-\varphi(*)\big)=\min_{A\in\A^\circ}\mu(\chi_A) \cdot \mu_A\big(\varphi|A-\varphi(*)\big)\le0.$$

The definition of the functional $p_X(\mu)$ implies that it is weakly additive. The weak multiplicativity of the functionals $\mu_A$, $A\in\A$, implies the weak multiplicativity of the functional $p_X(\mu)$. Next, we prove that the functional $p_X(\mu)$ weakly preserves order.
Indeed,
$$
\begin{aligned}
p_X(\mu)(\varphi)&=\varphi(*)+\sup_{A\in\U}\mu(\chi_A)\cdot\mu_A\big(\varphi|A-\varphi(*)\big)+
\inf_{A\in\U}\mu(\chi_A)\cdot\mu_A\big(\varphi|A-\varphi(*)\big)\ge\\
&\ge \varphi(*)+0+\inf_{A\in \U}\mu(\chi_A)\big(\min \varphi|A-\varphi(*)\big)\ge \min\varphi|A\ge\min\varphi.
\end{aligned}
$$By analogy we can prove that $p_X(\mu)(\varphi)\le\max\varphi$. So, $p_X(\mu)\in (V_\h\cap V_{\wa}\cap V_{\word})(X)$.

Finally, let us check that $p_X(\mu)\in \Lip_{2^{s+1}-1}(X)$. Fix any two functions $\varphi,\psi\in C(X)$ and find (possibly empty) sets $A^\varphi,A_\varphi,A^\psi,A_\psi\in \U$ such that
$$
\begin{aligned}
&\mu(\chi_{A^\varphi})\cdot\mu_{A^\varphi}(\varphi|A^\varphi-\varphi(*))=
\sup_{A\in\U}\mu(\chi_A)\cdot \mu_A(\varphi|A-\varphi(*)),\\
&\mu(\chi_{A_\varphi})\cdot\mu_{A_\varphi}(\varphi|A_\varphi-\varphi(*))=
\inf_{A\in\U}\mu(\chi_A)\cdot\mu_A(\varphi|A-\varphi(*)),\\
&\mu(\chi_{A^\psi})\cdot\mu_{A^\psi}(\psi|A^\psi-\psi(*))=
\sup_{A\in\U}\mu(\chi_A)\cdot\mu_A(\psi|A-\psi(*)),\\
&\mu(\chi_{A_\psi})\cdot\mu_{A_\psi}(\psi|A_\psi-\psi(*))=
\inf_{A\in\U}\mu(\chi_A)\cdot\mu_A(\psi|A-\psi(*)).
\end{aligned}
$$
The weak additivity and the $(2^{s}-1)$-Lipschitz property of the functionals $\mu_A$, $A\in\A$, imply:
$$
\begin{aligned}
p_X(\mu)(\varphi)-p_X(\mu)(\psi)&=\varphi(*)+
\mu(\chi_{A^\varphi})\cdot\mu_{A^\varphi}(\varphi|A^\varphi-\varphi(*))+
\mu(\chi_{A_\varphi})\cdot\mu_{A_\varphi}(\varphi|A_\varphi-\varphi(*))-\\
&\;-\psi(*)-\mu(\chi_{A^\psi})\cdot\mu_{A^\psi}(\psi|A^\psi-\psi(*))-
\mu(\chi_{A_\psi})\cdot\mu_{A_\psi}(\psi|A_\psi-\psi(*))\le\\
&\le\varphi(*)+\mu(\chi_{A^\varphi})\cdot\mu_{A^\varphi}(\varphi|A^\varphi-\varphi(*))+
\mu(\chi_{A_\psi})\cdot\mu_{A_\psi}(\varphi|A_\psi-\varphi(*))-\\
&\;-\psi(*)-\mu(\chi_{A^\varphi})\cdot\mu_{A^\varphi}(\psi|A^\varphi-\psi(*))-
\mu(\chi_{A_\psi})\cdot\mu_{A_\psi}(\psi|A_\psi-\psi(*))=\\
&=(1-\mu(\chi_{A^\varphi})-\mu(\chi_{A_\psi}))(\varphi(*)-\psi(*))+\\
&\;+\mu(\chi_{A^\varphi})\cdot(\mu_{A^\varphi}(\varphi|A^\varphi)-\mu_{A^\varphi}(\psi|A^\varphi))+\\
&\;+\mu(\chi_{A_\psi})\cdot(\mu_{A_\psi}(\varphi|A_\psi)-\mu_{A_\psi}(\psi|A_\psi))\le\\
&\le(1-\mu(\chi_{A^\varphi})-\mu(\chi_{A_\psi}))\|\varphi-\psi\|+
(2^{s}-1)\|\varphi-\psi\|+(2^{s}-1)\|\varphi-\psi\|\le\\
&\le(1+2(2^{s}-1))\|\varphi-\psi\|=(2^{s+1}-1)\|\varphi-\psi\|. \end{aligned}
$$
By analogy we can prove that $p_X(\mu)(\psi)-p_X(\mu)(\varphi)\le (2^{s+1}-1)\|\varphi-\psi\|$, which implies that the functional $p_X(\mu)$ is $(2^{s+1}-1)$-Lipschitz and hence belongs to $\Lip_{2^{s+1}-1}(X)$.

Using the continuity of the maps $p_A$, $A\in\A$, and the fact that for any $\e>0$ the norm $\|\varphi|A-\varphi(*)\|<\e$ for all but finitely many sets $A\in\A$, we can  show that the map $p_X:VX\to\overline{V}(X)$ is continuous.

Finally, we check that $p_X(\mu)=\mu$ if $\mu=\delta_X(x)$ is the Dirac measure at a point $x\in X$. If $x=*$, then $\mu(\chi_A)=\chi_A(*)=0$ and hence $p_X(\mu)(\varphi)=\varphi(*)=\mu(\varphi)$ for any function $\varphi\in C(X)$. So, we assume that $x\in X\setminus \{*\}$. Since $\A$ is a disjoint cover of $X$, there is a unique set $A\in\A$ containing $x$. It follows that $r_A(x)=x$ and hence $\mu_A=p_A\circ Vr_A(\mu)=p_A\circ Vr_A(\delta_X(x))=p_A(\delta_A(x))=\delta_A(x)$.  Take any function $\varphi\in C(X)$. If $\varphi(x)\ge \varphi(*)$, then
$$p_X(\mu)(\varphi)=\varphi(*)+\mu(\chi_A)\cdot\mu_A(\varphi|A-\varphi(*))+0=\varphi(*)+1\cdot\delta_A(x)(\varphi|A-\varphi(*))=
\varphi(*)+(\varphi(x)-\varphi(*))=\varphi(x).$$
If $\varphi(x)\le\varphi(*)$, then
$$p_X(\mu)(\varphi)=\varphi(*)+0+\mu(\chi_A)\cdot\mu_A(\varphi|A-\varphi(*))=
\varphi(*)+1\cdot(\varphi(x)-\varphi(*))=\varphi(x).$$
So, $p_X(\mu)=p_X(\delta_X(x))=\delta_X(x)=\mu$.
\end{proof}

\begin{theorem}\label{main2a} Each hereditarily paracompact scattered compact space $X$ of finite scattered height $s=\sht(X)$ is an absolute $F$-valued retract for the functor $F_{s+1}=V_{\h}\cap V_{\wa}\cap V_{\word}\cap\Lip_{2^{s+2}-1}$.
\end{theorem}

\begin{proof} Consider the one-point compactification $Y$ of the product $X\times\IN$ of $X$ with the countable discrete space $\IN$ and observe that $\sht(Y)=\sht(X)+1=s+1$ and $Y^{(s+1)}$ is a singleton. By Theorem~\ref{main2}, the space $Y$ is an absolute $F$-valued retract for the functor $F_{s+1}=V_{\h}\cap V_{\wa}\cap V_{\word}\cap\Lip_{2^{s+2}-1}$ and so is the space $X$, being homeomorphic to a retract of $Y$.
\end{proof}

It is known \cite[1.7]{SCH} that openly generated compacta have countable cellularity. This implies that any uncountable compact space $X$ with a unique non-isolated point is not openly generated. On the other hand, it is an absolute $F$-valued retract for the functor $F=V_\h\cap V_\wa\cap V_\word\cap \Lip_3$ according to Theorem~\ref{main2}. Let us write this fact for future references.

\begin{corollary} Any uncountable compact space $X$ with a unique non-isolated point is not openly generated but is an absolute $F$-valued retract for the functor $F=V_\h\cap V_\wa\cap V_\word\cap \Lip_3$. For this functor we get $\AR[F]\not\subset\OG$.
\end{corollary}

Theorems~\ref{main1}, \ref{main2} and Proposition~\ref{natural} have an interesting corollary.

\begin{corollary} The functor $V_\h\cap V_\wa\cap V_{\word}\cap\Lip_3$ admits no natural transformation into the functor $\Lip_k$ for any $k\in[1,2)$.
\end{corollary}

By the same reason, Theorems~\ref{shchep} and \ref{valovS} imply:

\begin{corollary} The functor $S=V_\h\cap V_\wa\cap V_\ord$ admits no natural transformation into a normal functor $F:\Comp\to\Comp$.
\end{corollary}

It is interesting to note that the functor $F=V_\h\cap V_{\wa}\cap V_{\word}$ has maximal possible class $\AR[F]$ of absolute $F$-valued retracts.

\begin{theorem}\label{all} For the functor $F=V_\h\cap V_{\wa}\cap V_{\word}$ the class $\AR[F]$ coincides with the class of all compact Hausdorff spaces.
\end{theorem}

\begin{proof} Observe that for every compact space $X$ the space $V_\word X=VX=\prod_{\varphi\in C(X)}[\min\varphi,\max\varphi]$ is an absolute retract, which implies that the class $\AR[V]$ contains all compact spaces. To prove that $\AR[F]=\AR[V]$ it suffices to construct a retraction $r_X:VX\to FX$.

Consider the functional $\alpha\in FX$ assigning to each function $\varphi\in C(X)$ the real number $\alpha(\varphi)=\frac12(\min \varphi+\max\varphi)$. Define the retraction $r_X:VX\to FX$ assigning to each  functional $\mu\in V_\word(X)$ the functional $p_X(\mu)$ assigning to each non-constant function $\varphi\in C(X)$ the real number
$$\alpha(\varphi)+\frac12\|\varphi-\alpha(\varphi)\|\cdot\bigg(\mu\Big(\frac{\varphi-\alpha(\varphi)}
{\|\varphi-\alpha(\varphi)\|}\Big)-
\mu\Big(-\frac{\varphi-\alpha(\varphi)}{\|\varphi-\alpha(\varphi)\|}\Big)\bigg).$$
It can be shown that $r_X(\mu)\in FX$ and the map $r_X:VX\to FX$ is a well-defined retraction of $VX$ onto $FX$.
\end{proof}

\section{Some Open Problems}

\begin{problem} Is the functor $F=V_\h\cap V_\wa\cap V_{\word}$ a (natural) retract of the functor $V_\word$?
\end{problem}

It seems that the retractions $r_X:VX\to FX$ constructed in the proof of Theorem~\ref{all} do not determine a natural transformation $r:V\to F$.

\subsection{Absolute $\Lip_k$-valued retracts}

\begin{problem} Is any compact space $X$ with a unique non-isolated point an absolute $\Lip_2$-valued retract?
\end{problem}

\begin{problem} Is each (scattered) compact space an absolute $\Lip_k$-valued retract for some $k$?
\end{problem}

The answer to this problem is unknown even for the Mr\'owka space $\psi_\A(\IN)$ generated by an uncountable almost disjoint family $\A$ of infinite subsets of $\IN$. By definition, $\psi_\A(\IN)$ is the Stone space of the Boolean algebra generated by $\A\cup\{\{n\}:n\in\IN\}$. In other words,
$\psi_\A(\IN)$ is the one-point compactification of the locally compact space $\IN\cup\A$ in which all points $n\in\IN$ are isolated and for any $A\in\A$ the family
$$\mathcal B_A=\big\{\{A\}\cup A\setminus F:\mbox{$F$ is a finite subset of $\IN$}\}$$
is a neighborhood base at $A$.
The Mr\'owka space has scattered height $\sht(\psi_\A(\IN))=2$ but is not hereditarily paracompact.
It is separable but contains an uncountable discrete subspace.
If the almost disjoint family $\A$ is maximal, then $\psi_\A(\IN)$ is sequential but not Fr\'echet-Urysohn.

\begin{problem} Is the Mr\'owka space $\psi_\A(\IN)$ an absolute $\Lip_k$-valued retract for some real $k\ge 1$?
\end{problem}

We say that a functor $F:\Comp\to\Comp$ is {\em weight-preserving} if for any infinite compact space $X$ the weight of the space $FX$ coincides with the weight of $X$.

\begin{problem} Can each compact space be absolute $F$-valued retract for some weight-preserving subfunctor $F\subset \overline{V}$? Is is true for the functor $\Lip_3$?
\end{problem}

\begin{problem} Is the Stone-\v Cech compactification $\beta\IN$ of positive integers an absolute $\Lip_3$-valued retract? Is the remainder $\beta\IN\setminus\IN$ of $\beta\IN$ an absolute $\Lip_3$-valued retract?
\end{problem}

\subsection{Absolute $F$-valued retracts for functors with finite supports}

A functor $F:\Comp\to\Comp$ is defined to have {\em finite supports} if for each compact space $X$ and an element $a\in FX$ there is a map $f:A\to X$ for a finite space $A$ such that $a\in Ff(FA)\subset FX$.

\begin{problem} Can a functor $F:\Comp\to\Comp$ with finite supports have $\AR[F]\supset \AbE[0]$?
\end{problem}

For a functor $F:\Comp\to\Comp$ and a natural number $n=\{0,\dots,n-1\}$ let $F_n$ be a subfunctor of $F$ assigning to each compact space $X$ the subspace
$$F_nX=\{a\in FX:\exists f\in C(n,X)\mbox{ such that } a\in Ff(Fn)\}.$$

\begin{problem} Given a functor $F:\Comp\to\Comp$ and a natural number $n$ study the class $\AR[F_n]$.
\end{problem}

The class of absolute $P_n$-valued retracts has been studied in \cite{Cat98} where it was observed that each $n$-dimensional compact metrizable space is an absolute $P_{n+2}$-valued retract and was proved that each $n$-dimensional hereditarily indecomposable compact space is not an absolute $P_n$-valued retract.

\subsection{Openly generated compacta and the superextension functor}

The following problem was motivated by Ivanov's Theorem~\ref{ivanov}.

\begin{problem} Let $F:\Comp\to\Comp$ be a monadic functional functor such that each connected openly generated compact Hausdorff space is an absolute $F$-valued retract. Is there a natural transformation $\lambda\to F$ from the functor of superextension $\lambda$ into $F$?
\end{problem}

\end{document}